\documentclass[11pt]{amsart}
\usepackage{geometry}   
\usepackage{graphicx}
\usepackage{amsmath}
\usepackage{amssymb, amsthm}
\usepackage{epstopdf}
 \newtheorem{theorem}{Theorem}
    \newtheorem{corollary}[theorem]{Corollary}

    \theoremstyle{definition}

  \newcommand{\Z}{\ensuremath{{\mathbb{Z}}}}

\newcommand{\lk}{{\rm {lk}}}
\newcommand{\st}{{\rm {st}}}
  \newcommand{\E}{\ensuremath{{\mathbb{E}}}}
  \renewcommand{\P}{\ensuremath{{\rm{Prob}}}}

\newcommand{\RAAG}{right-angled Artin group}
\newcommand{\RAAGs}{right-angled Artin groups}

\newcommand{\G}{\Gamma}
\newcommand{\XG}{X_\Gamma}

\newcommand{\AG}{A_\G}            
   
\newcommand{\GpS}{\ensuremath{{\mathcal{GS}_\Gamma}}}  
\newcommand{\Gp}{\ensuremath{{\mathcal{G}_\Gamma}}}  
\newcommand{\SG}{\ensuremath{{\mathcal{S}_\Gamma}}}  

\title{Random groups arising as graph products}       
\author{Ruth Charney and Michael Farber}   
\thanks {R. Charney was partially supported by NSF grant DMS 0705396.}
\thanks{M. Farber was partially supported by a grant from the EPSRC}

\begin{document}


\begin{abstract} In this paper we study the hyperbolicity properties of a class of random groups arising as graph products associated to random graphs. Recall, that the construction of a graph product is a generalization of the constructions of right-angled Artin and Coxeter groups. 
We adopt the  Erd{\H{o}}s - R\'{e}nyi model of a random graph and find precise threshold functions for 
the hyperbolicity (or relative hyperbolicity).
We aslo study automorphism groups of right-angled Artin groups associated to random graphs. 
We show that with probability tending to one as $n\to \infty$, random right-angled Artin groups have finite outer automorphism groups, 
assuming that the probability parameter $p$ is constant and satisfies $0.2929 <p<1$.
\end{abstract}

\maketitle

\section{Introduction}
The needs of mathematical modeling of large systems of various nature 
raise the problem of studying random geometric and algebraic objects. 
For a system of great complexity it is unrealistic to assume that 
one is able to have a precise description of its configuration space; the latter should be viewed as being a partially known or a random space. 
%
%
%

The most developed stochastic--topological object is a random graph. 
The theory of random graphs, initiated in 1959 by Erd\"os and R\'enyi \cite{ER}, is nowadays a well-developed and fast growing branch of discrete mathematics. The theory of random graphs 
\cite{AS}, \cite{B}, \cite{JLR} offers a plethora of spectacular results and predictions playing an essential role in various engineering and computer science applications.

Random simplical complexes of high dimension were recently suggested and studied by Linial--Meshulam in~\cite{LM}, and 
Meshulam--Wallach in~\cite{MW}. The fundamental groups of random 2-complexes are random groups of a fairly general type.

Configuration spaces of mechanical linkages with bars of random  lengths were studied in 
\cite{F1}, \cite{FK}. These are closed smooth manifolds depending on a large number of independent random parameters. Although the number of homeomorphism type of these manifolds grows extremely fast with the dimension, their topological characteristics can be predicted with high probability when the number of links tends to infinity. 

The theory of random groups of M. Gromov \cite{Gr}, \cite{G} has a density parameter $0\le d\le 1$, and it is known that for
$d<1/2$ a random group is infinite and hyperbolic, while it is trivial for $d>1/2$, see \cite{Gr}, page 273.
We refer to \cite{Zuk} and \cite{Oll} for more details.

Random right-angled Artin groups, which were first studied in \cite{CF}, represent a different class of random groups. 
In Gromov's model a random group is given by a presentation with randomly generated relations where each of the relations has a fixed length 
and the total number of relations is also fixed (it depends on the density $d$). In the case of random right-angled Artin groups one generates  randomly the commutation relations between the generators, and each of these relations appears with a fixed probability $0\le p\le 1$. 

Other types of probabilistic questions in group theory, such as the existence of homomorphisms between groups of a specific type, have been investigated by Shalev, Liebeck, and others (see for example, \cite{Sh05}).

In this paper we study a class of groups, called graph products, that includes both right-angled Artin groups and right-angled Coxeter groups. The paper has two main goals.  First, we study hyperbolicity of graph products associated to random graphs. 
Second, we study automorphisms of random right-angled Artin groups and more generally, random graph products of cyclic groups.

The interest in these results lies not in the random graph theory itself, but in the applications to random group theory.  Graph products form a large and interesting class of groups which, as this paper demonstrates, lend themselves naturally to probabilistic methods. 

Both authors would like to thank the Forschungsinstitut f\"ur Mathematik in Zurich for their hospitality during the writing of this paper.


\section{The hyperbolicity of random graph products}

The construction  of graph product is a generalization of the constructions of right-angled Artin and Coxeter groups. 
Let $\G$ be a finite simplicial graph with vertex set $V$ and let $\{G_v\}$ be a collection of non-trivial, finitely generated groups indexed by $V$.  Then the \emph{graph product}  $\Gp =\mathcal{G}(\G, \{G_v\})$ is the quotient of the free product $\prod G_v$ by commutator relations between the generators of $G_v$ and $G_w$ whenever $v,w$ are connected by an edge in $\G$.  Well known examples of these groups are the right-angled Coxeter groups (when $G_v = \Z / 2\Z$ for all $v$) and the right-angled Artin groups (when $G_v = \Z $ for all $v$).  We will denote these groups by $W_\G$ and $A_\G$ respectively. 

We use the Erd{\H{o}}s - R\'{e}nyi model of random graphs in which each edge of the complete graph on $n$ vertices is included with probability $0<p<1$ independently of all other edges, where $p$ is a function of $n$. In other words, we consider the probability space $G(n,p)$ of all $2^{\binom n 2}$ subgraphs of the complete graph on $n$ vertices $\{1, 2, \dots,n\}$ and the probability that a specific graph $\Gamma\in G(n,p)$ appears as a result of a random process equals
\begin{eqnarray}\label{prob}
{\P}(\Gamma)\,  = \, p^{E_\Gamma}(1-p)^{{\binom n 2}-E_\Gamma},
\end{eqnarray}
where $E_\Gamma$ denotes the number of edges of $\Gamma$, see \cite{JLR}.

We say a statement holds \emph{asymptotically almost surely} or \emph {a.a.s.} if the probability that it holds tends to one as $n \to \infty$.  
For $f,g$ functions of $n$ and $c$ a constant, we use the notation
\begin{itemize}
\item  $f \to c$ to mean $f$ tends to $c$ as $n$ goes to infinity, and
\item  $f \sim g$ to mean $f/g \to 1$.
\end{itemize}

Our main results concerning hyperbolicity are as follows:

 \begin{theorem}\label{thm1} Fix a collection of non-trivial, finite groups indexed by the natural numbers $\{G_i\}$. For a random graph  $\Gamma\in G(n,p)$,  let $\Gp = \mathcal G(\G, \{G_i\})$ be the associated graph group.  Then
\begin{enumerate}
\item If $(1-p)n^2 \to 0$ then $\Gp$ is finite, a.a.s
\item If $p n \to 0$ then $\Gp$ is hyperbolic, a.a.s.
\item If $pn \to \infty$ and $(1-p)n^2 \to \infty$, then $\G$ is not hyperbolic, a.a.s. 
\end{enumerate}
\end{theorem}

Applying this to the case where all vertex groups are cyclic of order 2, we obtain:

\begin{corollary} The right-angled Coxeter group $W_\Gamma$ corresponding to a random Erd{\H{o}}s - R\'{e}nyi
graph $\Gamma\in G(n, p)$ is hyperbolic (a.a.s) 
if either 
$p(1-n)^{2}\to 0$ or $pn \to 0$. If however $pn \to \infty$ and  $(1-p)n^2\to \infty$ 
then the random right-angled Coxeter group $W_\Gamma$ is not hyperbolic, a.a.s.
\end{corollary}

If the vertex groups $G_v$ are not necessarily finite, one may speak about hyperbolicity relative to the family of vertex subgroups $\{G_v\}$.  
We establish the following result:

\begin{theorem}\label{thm2} Fix a collection of non-trivial, finitely generated groups indexed by the natural numbers $\{G_i\}$. For a random graph  $\Gamma\in G(n,p)$,  let $\Gp = \mathcal G(\G, \{G_i\})$ be the associated graph group.  Then
\begin{enumerate}
\item If $(1-p)n^2 \to 0$ then $\Gp$ is isomorphic to the direct product $G_1 \times \dots \times G_n$, a.a.s.
\item If $pn\to 0$ then $\Gp$ is weakly hyperbolic relative to the family of subgroups $\{G_i\}_{i \leq n}$, a.a.s.
\item If $pn \to \infty$ and  $(1-p)n^2\to \infty$ then $\Gp$ is not weakly hyperbolic relative to the family of subgroups $\{G_i\}_{i \leq n}$, a.a.s.
\end{enumerate}
\end{theorem}

Since the vertex groups in Theorem \ref{thm2} are arbitrary finitely generated groups, one could add an additional random element by taking them to be some class of random groups.  The same result still holds.  Likewise in Theorem \ref{thm1}, the vertex groups can be random finite groups.

The proofs of Theorems \ref{thm1} and \ref{thm2} will be given in \S \ref{proofs1}.


\section{Graph products and buildings}  

Let $\Gp$ be a graph product of groups $\{G_v\}$ where $v$ runs over vertices of a graph $\Gamma$.  
We call a complete subgraph of $\G$ a \emph{clique}.\footnote{Some author's reserve the word clique for maximal complete subgraphs.  We do not require maximality.} We can associate a simplicial complex $\XG=X(\G, \{G_v\})$ with $\Gp$ as follows.  Define two posets, partially  ordered by inclusion, 
\begin{align*}
\SG &= \{ G_T \mid T \subseteq V, ~\text{$T=\emptyset$ or $T$ spans a clique in $\G$}\}\\
\GpS&= \{ gG_T \mid g \in \Gp, ~G_T \in \SG \}.
\end{align*}

Recall that the geometric realization (or flag complex) associated to a poset $\mathcal P$ is the simplicial complex whose vertices are the elements of $\mathcal P$ and whose $k$-simplices are totally ordered subsets $(p_0 < p_1 < \dots < p_k)$.
Let $X_\G$ be the geometric realization of $\GpS$ and let $K \subset X_\G$ be the geometric realization of $\SG$. 
Left multiplication induces an action of $\Gp$ on $X_\G$ with fundamental domain $K$.  The stabilizer of the vertex $gG_T$ is conjugate to $G_T$.  Thus, the action of $\Gp$ on $X_\G$ is always cocompact,
and it is proper if and only if all vertex groups are finite. 

In the case that every vertex group is cyclic of order 2, $\XG$ is the well-known Davis complex for the right-angled Coxeter group $W_\G$.  For a more general graph product $\Gp$, $\XG$ is a right-angled buiilding whose apartments are isomorphic to this Davis complex.  (See \cite{Dav} or  \cite{CRSV} for a discussion of these buildings.) 

The complexes $\XG$ have a natural metric.  Though $\XG$ was defined as a simplicial complex,
it also has a natural cubical structure. To see this, it suffices to describe the cubical structure on the fundamental domain $K$.  For a pair  $G_T \subseteq G_{T'}$ in  $\SG$, let
$$[G_T,G_{T'}]= \{ G_R \in \SG \mid T \subseteq R \subseteq T' \}.$$
It is easily seen that $[G_T,G_{T'}]$ spans a cube of dimension $|T'|-|T|$ in $K$.  
The cubical structure induces a piecewise Euclidean metric on  $\XG$ which was shown by M.~Davis to be CAT(0) \cite{Dav}.

Moussong \cite{Mou} showed that in some cases, the Davis complex could be given a CAT(-1) metric and used this to find precise conditions on when an arbitrary Coxeter group is word hyperbolic.  In the case of right-angled Coxeter groups, his conditions reduce to the requirement that $\G$ does not contain an \emph{empty square}, that is, a 4-cycle such that neither diagonal spans an edge in $\G$.  
This theorem was generalized by J.~Meier to graph products of finite groups.

\begin{theorem}[\cite{Mei}] \label{hyperbolic} Let $\Gp = \mathcal G(\G,\{G_v\})$ be a graph group where each $G_v$ is a non-trivial, finite group. Then $\Gp$ is Gromov hyperbolic if and only if $\G$ has no empty squares.
\end{theorem}

To do this, Meier (following Moussong) proves that if $\G$ has no empty squares, then the cubical metric on $\XG$ can be deformed to a CAT(-1) metric.  Conversely, if  $\G$ has an empty square, it is easy to show that the apartments in $\XG$ contain 2-flats, so $\XG$ cannot be hyperbolic.  Thus, $\XG$ is hyperbolic if and only if $\G$ has no empty squares. 
If the vertex groups are all finite, then the action of $\Gp$ on $\XG$ is proper and cocompact, so by the Milnor-Svarc lemma, $\Gp$ is quasi-isometric to $\XG$ and Meier's theorem follows.

If the vertex groups are not required to be finite, one can still draw some conclusions about $\Gp$, albeit weaker ones. While the action is no longer proper, it is discontinuous, that is,  the orbit of every point is discrete.   In \cite{CC}, the first author and J.~Crisp prove  a relative version of the Milnor-Svarc Lemma.  It states that if a finitely generated group $G$ acts discontinuously and cocompactly on a geodesic metric space $X$ with stabilizers conjugate to a collection of subgroups $\mathcal H$, then the graph obtained from the Cayley graph of $G$ by coning off all cosets of the subgroups in $\mathcal H$ is quasi-isometric to $X$.  A group $G$ is said to be \emph{weakly hyperbolic relative to $\mathcal H$} if this coned off Cayley graph is hyperbolic (see \cite{Far}).  Applying this to $\Gp$ acting on $\XG$, we see that $\Gp$ is weakly hyperbolic relative to  $\{G_v\}$ if and only if $\XG$ is hyperbolic. 

\begin{theorem}\label{relhyperbolic} Let $\Gp = \mathcal G(\G,\{G_v\})$ be a graph group where each $G_v$ is a non-trivial, finitely generated group.
Then $\Gp$ is weakly hyperbolic relative to $\{G_v\}$ if and only if $\G$ has no empty squares.
\end{theorem}

We remark that for a group to be (strongly) relatively hyperbolic requires an additional condition on quasi-geodesics in the Cayley graph.  In general, one does not expect this additional condition to hold for $\Gp$. Indeed, Behrstock, Drutu, and Mosher \cite{BDM} have shown that right-angled Artin groups associated to connected graphs $\G$ are not (strongly) relatively hyperbolic with respect to \emph{any} collection of proper subgroups.


\section{Proofs of Theorems \ref{thm1} and \ref{thm2}.}\label{proofs1}

The discussion in the previous section shows that the hyperbolicity (or relative hyperbolicity) of a graph group depends only on the existence of empty squares in the graph.  In this section, we consider the probability that a random graph contains an empty square.

The goal of this section is to establish threshold functions for the existence of an empty square in a random graph $\G$.  We will prove

\begin{theorem}\label{random}  Let $\Gamma \in G(n, p)$ be a random graph.
\begin{enumerate}
\item If $(1-p)n^2 \to 0$ then $\G$ is a complete graph, a.a.s.
\item If $pn \to 0$ then $\G$ has no empty squares, a.a.s.
\item If $pn\to \infty$ and $(1-p)n^2 \to \infty$ then $\G$ has an empty square, a.a.s.
\end{enumerate}
\end{theorem}

\begin{proof} The first statement is easy to prove. Namely, for a random graph $\Gamma \in G(n, p)$,  the probability that a given pair of vertices is not connected by an edge is $1-p$, hence the expected number of missing edges in $\G$ is
$${\binom {n} 2} (1-p) \sim \frac{1}{2}n^2(1-p).$$ 
By the first moment method (see p. 54 of \cite{JLR}),
 if this expectation goes to 0, then the probability that there exists a missing edge also goes to 0.
 Thus $n^2(1-p) \to 0$ implies that
$\Gamma$ is a complete graph with probability tending to one as $n\to \infty$. 

For the second statement, fix a set of $n$ vertices and consider an ordered 4-tuple of distinct vertices $(v_1,v_2,w_1,w_2)$.  Let $I_{(v_1,v_2,w_1,w_2)}: G(n,p) \to \{0,1\}$ be the random variable which takes the value $1$ 
on $\Gamma \in G(n,p)$ 
if and only if $(v_1,v_2,w_1,w_2)$ span an empty square in $\G$ with $\{v_1,v_2\}$ and $\{w_1,w_2\}$ appearing as diagonal pairs.  
\begin{figure}[h]
\begin{center}
\includegraphics[width=4cm]{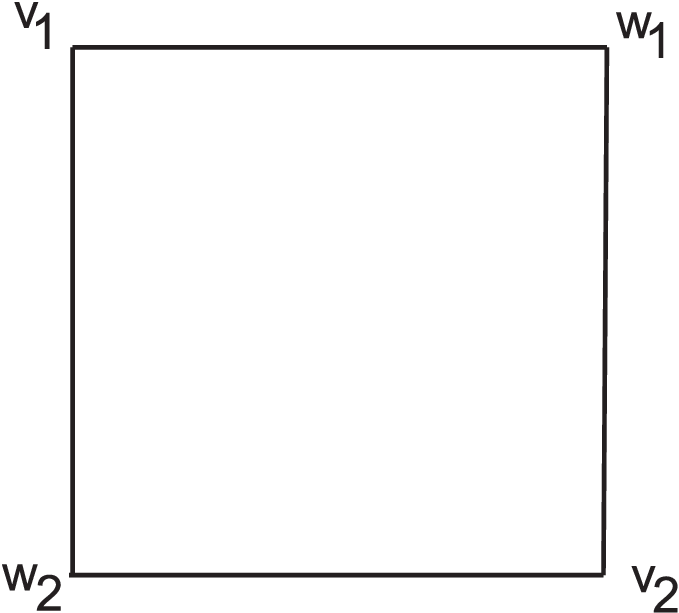}
\end{center}
\end{figure}
In more detail, one has $I_{(v_1,v_2,w_1,w_2)}(\Gamma)=1$ if and only if the edges $v_1w_1, v_1w_2, v_2w_1$ and $v_2w_2$ are included in $\Gamma$ and the edges 
$v_1v_2$ and $w_1w_2$ are not included in $\Gamma$. 
The sum
$$X=\sum I_{(v_1,v_2,w_1,w_2)}: G(n, p)\to \Z,$$
over all 4-tuples of vertices counts each empty square in $\G$ eight times, corresponding to the 8 reflections of the square.
The expectation $\E(I_{(v_1,v_2,w_1,w_2)})$ equals $p^4(1-p)^2$, hence
\begin{eqnarray*}
\E(X)&=& \sum  \E(I_{(v_1,v_2,w_1,w_2)})\\
 &=& n(n-1)(n-2)(n-3) p^4(1-p)^2 \\
 &\sim& n^4p^4(1-p)^2
 \end{eqnarray*}
If $pn \to 0$ as $n \to \infty$, then $(1-p) \to 1$ and $\E(X) \to 0$, hence the probability that $\G$ has an empty square goes to 0. This proves statement (2). 

For the third statement, we use the second moment method. Namely, the probability that $X \neq 0$ satisfies
$$ \P(X \neq 0) \geq \frac{(\E X)^2}{\E(X^2)} $$
(see \cite{JLR}, page 54).
Thus to prove (3), it suffices to show that $(\E X)^2 \sim \E(X^2)$.
By the previous paragraph, we have 
$$(\E X)^2 \sim n^8p^8(1-p)^4.$$
To compute $ \E(X^2)$, we divide the sum $$X^2= \sum I_{(v_1,v_2,w_1,w_2)}I_{(x_1,x_2,y_1,y_2)}$$ into several cases. Write $I_{\bf{v,w}}=I_{(v_1,v_2,w_1,w_2)}$ and $I_{\bf{x,y}}=I_{(x_1,x_2,y_1,y_2)}$.

Case (A).  Let $X_1$ denote the sum of all products $I_{\bf{v,w}}I_{\bf{x,y}}$ such that  no vertex appears in both 4-tuples $(v_1,v_2,w_1,w_2)$ and 
$(x_1,x_2,y_1,y_2)$. Then
$$\E(X_1) = \frac{n!}{(n-8)!} p^8(1-p)^4 \sim n^8 p^8(1-p)^4.$$

Case (B).  Let $X_2$ denote the sum over all $I_{\bf{v,w}}I_{\bf{x,y}}$ such that exactly one vertex appears in both 4-tuples.  Then
$$\E(X_2) = 16\frac{n!}{(n-7)!} p^8(1-p)^4 \sim 16 n^7p^8(1-p)^4$$

Case (C).   Let $X_3$ denote the sum over all $I_{\bf{v,w}}I_{\bf{x,y}}$ such that 2 vertices appear in both 4-tuples making one of the pairs $\bf{v}$ or $\bf{w}$ equal to one of the pairs $\bf{x}$ or $\bf{y}$ up to permutation, i.e., the potential squares share a pair of diagonal vertices.  In this case, 
$(I_{\bf{v,w}}I_{\bf{x,y}})$  depends on the existence of 8 sides and the non-existence of 3 diagonals, so
$$\E(X_3) = 8 \frac{n!}{(n-6)!} p^8(1-p)^3 \sim 8 n^6p^8(1-p)^3.$$

Case (D).   Let $X_4$ denote the sum over all $I_{\bf{v,w}}I_{\bf{x,y}}$ such that 2 vertices appear in both 4-tuples but in different diagonal pairs, i.e., the potential squares share an edge.  Then 
$(I_{\bf{v,w}}I_{\bf{x,y}})$  depends on the existence of 7 sides and the non-existence of 4 diagonals, so
$$\E(X_3) = 32 \frac{n!}{(n-6)!} p^7(1-p)^4 \sim 32 n^6p^7(1-p)^4.$$
Note that there is no need to consider the case in which some pair appears as a diagonal pair in one 4-tuple and an edge pair in the other since in that case either $I_{\bf{v,w}}$ or $I_{\bf{x,y}}$ must be 0.

Case (E).   Let $X_5$ denote the sum over all $I_{\bf{v,w}}I_{\bf{x,y}}$ such that the two 4-tuples share 3 vertices. Then
$$\E(X_5) = 16 \frac{n!}{(n-5)!} p^6(1-p)^3 \sim 16 n^5p^6(1-p)^3.$$

Case (F)  Let $X_5$ denote the sum over all $I_{\bf{v,w}}I_{\bf{x,y}}$ such that up to permutation, the two 4-tuples are the same. Then
$$\E(X_6) = 8 \frac{n!}{(n-4)!} p^4(1-p)^2 \sim 8 n^4p^4(1-p)^2.$$

Now compute
\begin{eqnarray}\label{long}
 \frac{\E(X^2)}{(\E X)^2} &=& \sum _{i=1}^{6}  \frac{\E(X_i)}{(\E X)^2} \nonumber \\
 &\sim & 1 + \frac{16}{n} +  \frac{8}{n^2(1-p)} + \frac{32}{n^2p} + \frac{16}{n^3p^2(1-p)} 
 +\frac{8}{n^4p^4(1-p)^2}.
 \end{eqnarray}
Note that since either $(1-p) \geq \frac{1}{2}$ or  $p \geq \frac{1}{2}$,
$$n^3p^2(1-p) \ge \min\{\frac{1}{2} n^3p^2, \frac{1}{4}n^3(1-p)\},$$
and
$$n^4p^2(1-p)^2 \ge \min\{\frac{1}{4}n^4p^4, \frac{1}{16} n^4(1-p)^2\}.$$
Thus, if $pn\to \infty$ and $(1-p)n^2\to \infty$, then all denominators in (\ref{long}) tend to infinity implying 
$$ \P(X \neq 0) \geq \frac{(\E X)^2}{\E(X^2)} \to 1.$$
This completes the proof of Theorem \ref{random}.
\end{proof}

 Theorems \ref{thm1} and \ref{thm2} now follow by combining Theorems \ref{hyperbolic}, \ref{relhyperbolic} and \ref{random}.


\section{Automorphism groups of random right-angled Artin groups}
 
A right-angled Artin group is a graph product in which vertex groups are infinite cyclic.  
In this section we review some basic properties of \RAAGs\  and refer the reader to \cite{HM95}, \cite{Ser89} and \cite{Lau95} for details.  For a general survey on these groups, see \cite{Ch07}.

To specify a \RAAG, we need only specify the graph $\G$.  If $V$ is the vertex set  of $\G$, then 
the \RAAG associated to $\G$ is the group with presentation
\[
\AG=\langle V \mid vw=wv \hbox{ if  $v$ and $w$ are connected by an edge in $\G$}\rangle.
\]

At one extreme, we have the case that $\G$ is a discrete graph (i.e., no edges) and $\AG$ is the free group on $V$.  At the other, is the case that $\G$ is a complete graph (i.e., every pair of vertices spans an edge) and $\AG$ is the free abelian group on $V$.

In this section, we consider right-angled Artin groups associated to random graphs $\Gamma\in G(n,p)$. 
It is clear that basic properties of $\Gamma$ are reflected in properties of the group.  For example,
$A_\Gamma$ decomposes as a free product if and only if $\Gamma$ is disconnected.  For $\Gamma\in G(n,p)$, Erd{\H{o}}s and R\'enyi \cite{ER} showed that this holds asymptotically almost surely if 
$p(n)=\frac{\log n -\omega (n)}{n}$, where $\omega : \mathbb N \to \mathbb R$ is a function such that $\lim_{n \to \infty} \omega (n) = \infty$.  Likewise,
$A_\Gamma$ decomposes as a direct product if and only if the complementary graph is disconnected, which  holds a.a.s. if  $1-p(n)=\frac{\log n -\omega (n)}{n}$.
In \cite{CF},  Costa and the second author analyze the cohomological dimension and the topological complexity of right-angled Artin groups associated to random graphs.

Automorphsim groups of right-angled Artin groups have been extensively studied in recent years.
(See, for example, \cite{CV08, CV09, Day09}.)  Here, we consider automorphism groups of right-angled Artin groups associated to random graphs.
Our goal is to prove the following theorem:

\begin{theorem}\label{thm11}
Let $\Gamma \in G(n, p)$ be a random graph where the probability parameter $p$ is independent on $n$ and satisfies
\begin{eqnarray}\label{parameter}1-\frac{1}{\sqrt{2}}< p < 1.
\end{eqnarray}
Then the right-angled Artin group $A_\Gamma$ determined by $\Gamma$ has a finite outer automorphism group ${\rm {Out}}(A_\Gamma)$, asymptotically almost surely. 
\end{theorem}

Note that $1-1/\sqrt{2}\sim 0.2929$. 

It would be interesting to know if Theorem \ref{thm11} holds outside the range (\ref{parameter}). Our proof shows that a random right-angled Artin group $A_\Gamma$ admits no transvections for any fixed $p$; however to exclude partial conjugations one needs the assumption (\ref{parameter}). The definitions of transvections and partial conjugations are given  below.


\section{Proof of Theorem \ref{thm11}}

The automorphism groups of right-angled Artin groups were first studied by Servatius \cite{Ser89} who described a finite generating set for $Aut(\AG)$ 
under certain restrictions on $\G$.  This was later extended to arbitrary $\G$ by Laurence in \cite{Lau95}.  
The generators depend on certain structures in the defining graph $\G$.  For a vertex $v$, the \emph{link} $\lk(v)$ is defined as the set of all vertices which are connected to $v$ by an edge. The \emph{star} $\st(V)$ is the union of all edges containing $v$.\footnote{In some contexts, the link and star are taken to be full subgraphs of $\G$.  In our context, only the vertices in these subgraphs will matter.}
The Servatius-Laurence generating set consists of the following automorphisms.

\begin{itemize}
  \item \emph{Symmetries}: These are given by permutations of the generators arising from symmetries of the graph $\G$.

   \item  \emph{Inversions}: These send a generator $v \in V$ to its inverse.

   \item  \emph{Transvections}:  These map   $v\mapsto vw$ where $v,w \in V$ satisfy   $\lk(v) \subseteq \st(w)$.
 
  \item   \emph{Partial conjugations}:  These conjugate all of the generators in one component of 
  $\G \backslash \st(v)$ by $v$ and occur only if $\G \backslash \st(v)$ is disconnected.
  
\item  \emph{Inner automorphisms}: These conjugate the entire group by some generator $v$.
\end{itemize}

Let us take a closer look at transvections.  The condition that $\lk(v) \subseteq \st(w)$
is necessary to guarantee that the map  $v\mapsto vw$  preserves commutator relations.  With this in mind, we introduce a partial order on $V$.  For two vertices $v,w$, write 
\[
\textrm{$v \leq w$ if $\lk(v) \subset \st(w)$.}
\]
  It is shown in \cite{CV08} that this relation is transitive.  

Taking the quotient of $Aut(\AG)$ by the inner automorphisms gives the outer automorphism group $Out(\AG)$.
The first two types of automorphisms, symmetries and inversions, induce a finite subgroup of 
$Aut(\AG)$ (and hence of $Out(\AG)$)  while transvections  and partial conjugations have infinite order.   Hence $Out(\AG)$ is finite if and only if $\AG$ does not permit any transvections or partial conjugations.  To exclude transvections, we require that no two vertices are related by the partial order $\leq$.  
Consider, for example, the case where $\Gamma$ consists of a single cycle of length $n \geq 5$.  It is easily seen that in this case, no two vertices satisfy $lk(v) \subset st(w)$.
 To exclude partial conjugations, we require that $\G - \st(v)$ is connected for every vertex $v$.  We call such a graph \emph{star 2-connected}.

 To prove Theorem \ref{thm11}, we will show that for $\Gamma\in G(n,p)$ with $p$ satisfying (\ref{parameter}), the probability that $A_\Gamma$  admits a transvection or a partial conjugation tends to zero as $n$ goes to infinity. 

First we show that the probability that there exists a pair of vertices $v,w$ in $\Gamma$ with $\lk(v) \subset \st(w)$ tends to zero as $n \to \infty$. This would imply that $A_\Gamma$ admits no transvections.

\begin{figure}[h]
\begin{center}
\includegraphics[width=3in]{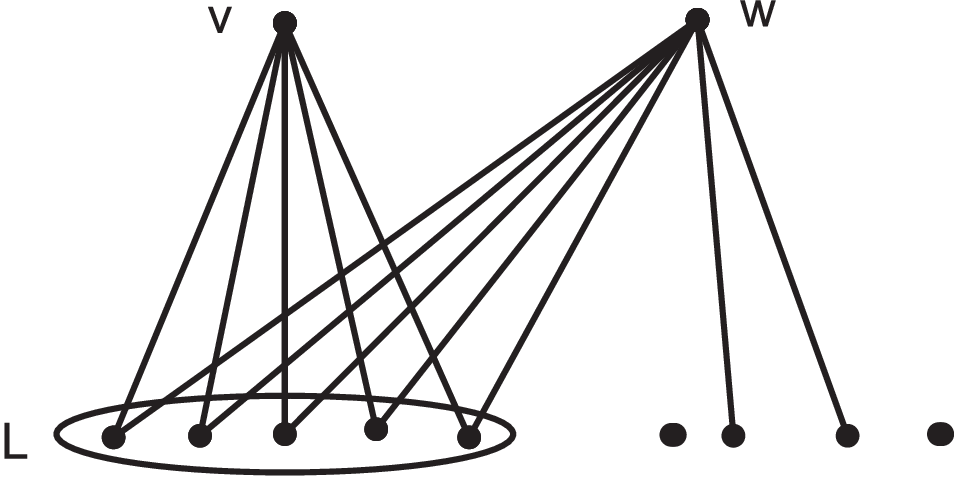}
\end{center}
\end{figure}

For a pair of distinct vertices $v, w\in \{1, \dots, n\}$ and for a subset $L\subset \{1, \dots, n\}$ not containing $v, w$, consider the following random variable 
$$I_{v,w, L}: G(n, p) \to \{0,1\},
$$
where for a graph $\Gamma\in G(n, p)$ one has $I_{v,w, L}(\Gamma)=1$ if and only if in $\Gamma$ the following conditions hold: (a) the link of $v$ equals either $L$ or $L\cup \{w\}$
and (b) $L\subset \lk(w)$; otherwise 
we set $I_{v,w, L}(\Gamma)=0$. 
In other words, one has $I_{v,w, L}(\Gamma)=1$ if and only if (1) every vertex of $L$ is connected by an edge to the vertices $v$ and $w$; (2) no vertex of $\{1, \dots, n\}-L-\{v, w\}$ is connected by an edge to $v$. (See figure.)


The sum 
$$X=\sum I_{v,w, L}: G(n, p)\to \Z,$$
where $v, w, L$ run over all possible choices, 
counts the number of ordered pairs of vertices $(v, w)$ with the property that $\lk(v) \subset \st(w)$. 

The expectation $\E(I_{v, w, L})$ equals 
$$p^k(1-p)^{n-2-k}p^k= p^{2k}(1-p)^{n-2-k}$$ where $k=|L|$.  
Thus, we obtain
\begin{eqnarray}
\E(X) &=& n(n-1) \sum_{k=0}^{n-2}{\binom {n-2} k} 
p^{2k}(1-p)^{n-2-k} \\
&=& n(n-1)(1-p+p^2)^{n-2}.\nonumber\end{eqnarray}

Now suppose that 
\begin{eqnarray}\label{condition}
p(1-p)n -2\log n \to \infty, \quad \mbox{as}\quad n\to \infty.\end{eqnarray}
This can also be expressed by saying that
$$p(1-p)=\frac{2\log n +\omega(n)}{n}$$
where $\omega(n) \to \infty$ as $n\to \infty$. Note that any constant $0<p<1$ satisfies this condition.
Then denoting $x=p-p^2$ one has
\begin{eqnarray*} \log \E(X) &\le& 2 \log n +(n-2) \log(1-x)  \\ &=& 2\log n - (n-2)\left[x+\frac{x^2}{2}+ \frac{x^3}{3}+\dots\right]\\ 
&=& 2\log n -nx +x\left[2-(n-2)\left[\frac{x}{2}+ \frac{x^2}{3} + \dots\right]\right]\\
&\le& 2\log n - nx= - \omega(n) .
\end{eqnarray*}
Thus, assuming (\ref{condition}), we have $\E(X)\to 0$ as $n\to \infty$.  By the first moment method,
$\P(X>0) \le \E(X)$,  so we conclude that ${\rm {Prob}}(X>0)\to 0$ as $n\to \infty$. 

Now we examine conditions for the absence of partial conjugations, i.e. the conditions on $\Gamma$ such that the result of removing the star of any vertex is path-connected. 

Consider a partition of the set of vertices $\{1, \dots, n\}$ into $v\cup L\cup S\cup T$ where $v$ is a vertex and $L, S, T$ are disjoint subsets not containing $v$. 
We denote $\ell=|L|, s=|S|, t=|T|$. Let $J_{v, L,S,T}: G(n, p)\to \{0,1\}$ be the random variable which associates to a random graph $\Gamma\in G(n, p)$ one if and only if $L$ 
is the link of $v$ in $\Gamma$ and there are no edges in $\Gamma$ connecting a point of $S$ to a point of $T$. One has 
$$\E(J_{v, L, S, T})= p^\ell(1-p)^{n-1-\ell}(1-p)^{st}.$$
Consider also the sum 
$$Y= \sum J_{v, L, S, T}$$ where the sum is taken over all possible choices of $v, L, S, T$ with $s\ge 1$ and $t\ge 1$. 
For a random graph $\Gamma\in G(n, p)$ the number $Y(\Gamma)$ is positive if and only if removing the star of some vertex disconnects the graph. Thus 
$\P(Y>0)$ is the probability that a random graph fails to be star 2-connected.

For the expectation of $Y$ we have
\begin{eqnarray}\label{four}
\E(Y) &=& n\cdot \sum_{\stackrel {s\ge 1, t\ge 1} {1+\ell+s+t=n}} \binom {n-1} \ell \cdot \binom {n-1-\ell} s \cdot p^\ell\cdot (1-p)^{n-\ell-1+st}\nonumber\\
&=& n(1-p)^{n-1} \cdot \sum_{\ell=0}^{n-3} {\binom {n-1} \ell} \cdot \left(\frac{p}{1-p}\right)^\ell \cdot F_\ell(p),\end{eqnarray}
where 
\begin{eqnarray}F_\ell(p) = 
\sum_{\stackrel {s+t=n-1-\ell}{s\ge 1, t\ge 1}}\binom {s+t} s\cdot (1-p)^{st}.
\end{eqnarray}

Observe that for $s\ge 2$, $ t\ge 2$ one has $st=s(n-1-\ell-s) \ge 2(n-3-\ell)$. Thus separating the terms $s=1, t=1$ gives
\begin{eqnarray*}
F_\ell(p) &=& 2(n-1-\ell)\cdot (1-p)^{n-2-\ell} + \sum_{\stackrel {s+t=n-1-\ell}{s\ge 2, t\ge 2}}\binom {s+t} s\cdot (1-p)^{st}\\
 &\le &
2(n-1-\ell)\cdot (1-p)^{n-2-\ell} + (1-p)^{2(n-3-\ell)}\cdot 2^{n-1-\ell} \\ \\
&\le & 2n \cdot (1-p)^{n-2-\ell}+ (1-p)^{-4} \left[2(1-p)^2\right]^{n-1-\ell}.\end{eqnarray*}

Substituting this into (\ref{four}) we find
\begin{eqnarray*}
\E(Y) &\le &2n^2(1-p)^{n-2}\cdot \sum_{\ell=0}^{n-3} \binom {n-1} \ell \left(\frac{p}{1-p}\right)^\ell\left(1-p\right)^{n-1-\ell}\\
& & +\quad n(1-p)^{n-5} \cdot \sum_{\ell+0}^{n-3} \binom {n-1} \ell \left(\frac{p}{1-p}\right)^\ell \left(2(1-p)^2\right)^{n-1-\ell}\\
&\le & 2n^2\cdot(1-p)^{n-2}\cdot \left[\frac{p}{1-p}+ 1-p\right]^{n-1} \\
& & + \quad n\cdot (1-p)^{n-5}\left[\frac{p}{1-p} +2(1-p)^2\right]^{n-1} \\
&=& 2n^2(1-p)^{-1}\left[p+(1-p)^2\right]^{n-1} + n(1-p)^{-4}\cdot [p+2(1-p)^3]^{n-1}\\ \\
&=& 2n^2 (1-p)^{-1} (1-x)^{n-1} + n(1-p)^{-4} (1- y)^{n-1},
\end{eqnarray*}
where $x= p-p^2$ and $y= 2(p-1)(p-1-\frac{\sqrt{2}}{2})(p-1+\frac{\sqrt{2}}{2})$. Clearly for all $0<p<1$, one has  $0<x<1$; and it is easy to see that for $1-\frac{1}{\sqrt{2}}<p<1$, one also has $0<y<1$. 

Thus we obtain that as $n\to \infty$,  the expectation $\E(Y)$ tends to zero and by the first moment method, this implies that the probability  $\P(Y>0)$ tends to zero.
Therefore, the right-angled Artin group corresponding to a random graph $\Gamma$ with parameter $p$ 
satisfying (\ref{parameter}) admits no partial conjugations, a.a.s. 

Since we have already proven that $A_\Gamma$ admits no transvections, it follows that, for a random graph
$\Gamma\in G(n,p)$ with $p$ satisfying (\ref{parameter}) 
the group $A_\Gamma$ has a finite outer automorphism group. 

This completes the proof of Theorem \ref{thm11}.


These results also apply to more general  graph products $\Gp =\mathcal{G}(\G, \{G_v\})$ in which every vertex group is cyclic, but not necessarily infinite cyclic.  Corredor and Gutierrez \cite{CG10} have recently shown that the automorphsim groups of such graph products have a generating set analogous to the Servatius-Laurence generators for a right-angled Artin group described above.  The only difference is that inversions are replaced by isomorphisms of individual vertex groups (these are always of finite order) and there are additional restrictions on transvections  involving the orders of the vertex groups.  Once again, if there are no vertices with $\lk(v) \subset st(w)$ and no separating stars, then the automorphism group is necessarily finite.  Thus we obtain,

\begin{theorem}
Fix a collection of non-trivial, cyclic groups indexed by the natural numbers $\{G_i\}$. For a random graph  $\Gamma\in G(n,p)$,  let $\Gp = \mathcal G(\G, \{G_i\})$ be the associated graph group.  Suppose 
the probability parameter $p$ is independent of $n$ and satisfies
\begin{eqnarray}\label{parameter2}1-\frac{1}{\sqrt{2}}< p < 1.
\end{eqnarray}
Then $\Gp$ has a finite outer automorphism group ${\rm {Out}}(A_\Gamma)$, asymptotically almost surely. 
\end{theorem}




\def\cprime{$\prime$}

\end{document}